\documentclass[letterpaper, 10pt, twocolumn]{ieeeconf}      

\IEEEoverridecommandlockouts                              
\usepackage{amsmath,mathrsfs,amsfonts,amssymb,graphicx,epsfig}
\usepackage[dvipsnames]{xcolor}
 \usepackage{amsthm}
\usepackage{subcaption}
\usepackage{caption}
\usepackage{color,multirow,rotating}
\usepackage{algorithm,algpseudocode,algorithmicx}
\usepackage{cite,url,framed,bm,balance,dsfont,varwidth}
\usepackage{nicefrac}

\newtheorem{theorem}{Theorem}

\newtheorem{remark}{Remark}

\usepackage{cleveref}
\usepackage{tikz}
\usetikzlibrary{calc,trees,positioning,arrows,chains,shapes.geometric,decorations.pathreplacing,decorations.pathmorphing,shapes, matrix,shapes.symbols}

\newcommand{\differential}{{\rm{d}}}

\newcommand{\RNum}[1]{\uppercase\expandafter{\romannumeral #1\relax}}

\allowdisplaybreaks

\title{\LARGE\textbf{
A Physics-informed Deep Learning Approach for \\ Minimum Effort Stochastic Control of Colloidal Self-Assembly}
}

\author{Iman Nodozi, Jared O'Leary, Ali Mesbah, Abhishek Halder
\thanks{Iman Nodozi is with the Department of Electrical and Computer Engineering, University of California, Santa Cruz, CA 95064, USA,
        {\tt\small{inodozi@ucsc.edu}}.\\
        Jared O'Leary and Ali Mesbah are with the Department of Chemical and Biomolecular Engineering, University of California, Berkeley, CA 94720, USA,
        {\tt\small{\{jared.oleary,mesbah\}@berkeley.edu}}.\\
        Abhishek Halder is with the Department of Applied Mathematics, University of California, Santa Cruz, CA 95064, USA,
        {\tt\small{ahalder@ucsc.edu}}.\\
        This work was partially supported by NSF grants 2112754 and 2112755.
}}

\begin{document}

\maketitle
\pagenumbering{arabic}

\bstctlcite{IEEEexample:BSTcontrol} 

\begin{abstract}
We propose formulating the finite-horizon stochastic optimal control problem for colloidal self-assembly in the space of probability density functions (PDFs) of the underlying state variables (namely, order parameters). The control objective is formulated in terms of steering the state PDFs from a prescribed initial probability measure towards a prescribed terminal probability measure with minimum control effort. For specificity, we use a univariate stochastic state model from the literature. Both the analysis and the computational steps for control synthesis as developed in this paper generalize for multivariate stochastic state dynamics given by generic nonlinear in state and non-affine in control models. We derive the conditions of optimality for the associated optimal control problem. This derivation yields a system of three coupled partial differential equations together with the boundary conditions at the initial and terminal times. The resulting system is a generalized instance of the so-called Schr\"{o}dinger bridge problem. We then determine the optimal control policy by training a physics-informed deep neural network, where the ``physics'' are the derived conditions of optimality. The performance of the proposed solution is demonstrated via numerical simulations on a benchmark colloidal self-assembly problem.
\end{abstract}

\section{Introduction}\label{sec:introduction}

 Colloidal self-assembly (SA) is the process by which discrete components (e.g., micro-/nano-particles in solution) spontaneously organize into an ordered state \cite{whitesides2002self}. The spontaneous self-organization central to colloidal SA enables ``bottom-up'' materials synthesis, which can allow for manufacturing advanced, highly-ordered crystalline structures in an inherently parallelizable and cost-effective manner \cite{paulson2015control, tang2022control}. The fact that colloidal SA can begin with micro- and/or nano-scale building blocks of varying complexity indicates that this bottom-up engineering approach can be used to synthesize novel metamaterials with unique optical, electrical, or mechanical properties \cite{juarez2012feedback, paulson2015control, tang2022control}.

Colloidal SA is an inherently stochastic (i.e., random) process prone to kinetic arrest due to particle Brownian motion \cite{paulson2015control, gillespie2007stochastic, gillespie2013perspective, tang2022control}. This leads to variability in materials manufacturing and possibly high defect rates, which can severely compromise the viability of using colloidal SA to reproducibly manufacture advanced materials. This lack of reproducibility in turn prevents colloidal SA from achieving cost-effective and scalable manufacturing of such materials \cite{paulson2015control, tang2022control, furst2013directed, liddle2016nanomanufacturing}. Thus, the thermodynamic and kinetic driving forces that govern colloidal SA, will need to be precisely and systematically modulated to consistently and efficiently direct colloidal SA systems towards high-value mass-producible structures and materials.

To more reproducibly drive collodial SA systems towards desired structures, it has been proposed to design a model-based feedback control policy wherein a global actuator (e.g., electric field voltage) is manipulated based on currently available information on the system state and a dynamical system model \cite{tang2013colloidal, tang2016optimal, xue2014optimal}. Work in \cite{tang2013colloidal} presents a model predictive control (MPC) method for controlling colloidal SA. These authors consider the dynamical model for the stochastic colloidal assembly process based on actuator-parameterized Langevin equations. In \cite{xue2014optimal, tang2016optimal}, the system is guided towards the desired highly-ordered structure based on a Markov decision process optimal control policy. 

In comparison, the perspective and approach taken in this paper towards optimal control of a colloidal SA process are significantly different, as we seek to control the time evolution of the joint probability distribution supported over the states of colloidal SA. Our technical approach and contributions are as follows.
\begin{itemize}
    \item[(1)] We show that the problem of controlling colloidal SA over a finite-time horizon can be naturally formulated in terms of steering of the probability distributions of the underlying stochastic system states, namely the order parameters. This leads to a two-point boundary value problem over the manifold of state probability measures, which from a control viewpoint is a non-traditional stochastic optimal control problem known as the Schr\"{o}dinger bridge problem (SBP), see e.g., \cite{chen2021stochastic}. The notion of lifting the stochastic control of colloidal SA directly onto the space of state probability measure as an SBP, is novel.
    
    \item[(2)] While there exists a growing literature on numerically solving the SBPs with nonlinear prior dynamics \cite{caluya2021wasserstein,caluya2021reflected,nodozi2022schr,caluya2019finite,caluya2020finite}, these works leverage specific structures of the underlying drift and diffusion terms in a stochastic differential equation (SDE). In contrast, the typical colloidal SA application, as considered here, requires more generic considerations since  the drift and diffusion coefficients can be nonlinear w.r.t to the states as well as non-affine in control. We show that, unlike the existing SBP conditions of optimality, our setting leads to a system of three coupled partial differential equations (PDEs) with endpoint PDF constraints. These three PDEs are the controlled Fokker-Planck-Kolmogorov (FPK) PDE, the Hamilton-Jacobi-Bellman (HJB) PDE, and a policy PDE.
    
    \item[(3)] The resulting system of three coupled PDEs is not amenable to existing computational approaches for the SBPs, such as the contractive fixed point recursions on the so-called Schr\"{o}dinger factors. To address this challenge, we employ the notion of physics-informed neural networks (PINN) (e.g., \cite{raissi2019physics,lu2021deepxde}) 
    to train a deep neural network approximating the solution of our coupled PDE system and the boundary conditions.
\end{itemize}

The distinct feature of the proposed control methodology is that it derives an optimal control policy that steers a given probability distribution of the order parameters of a colloidal SA system to a desired one over a finite-time horizon. The computation does not require making parametric approximation of the statistics such as the Gaussian mixture or exponential family. It also avoids approximating the nonlinearities in the drift and diffusion \emph{a priori}, e.g., via Taylor series.

Our technical contribution is a new control methodology. To make the exposition concrete, we use a specific univariate SDE model from the literature \cite{tang2013colloidal,xue2014optimal} to demonstrate the proof-of-concept for our proposed method. The specificity of the model is used from the outset as a didactic writing style, even though our proposed method is general, i.e., \emph{not} contingent on the nonlinearity structure or dimensionality of the SDE model.

\subsubsection*{Organization} This paper is structured as follows. Sec. \ref{Sec:ProblemFormulation} details our proposed stochastic optimal control problem formulation. In Sec. \ref{secOptimality}, we derive the first order conditions for optimality in the form of a system of coupled PDEs. We then learn the solutions for this system of equations by training a PINN, as detailed in Sec. \ref{secPINN}. The numerical simulation results are reported in Sec. \ref{sec:SimulationResults}, followed by concluding remarks in Sec. \ref{sec:conclusions}.


\subsubsection*{Notations} We use $\mathbb{E}_{\mu^{\pi}}\left[\cdot\right]$ to denote the expectation w.r.t. the controlled state probability measure or distribution $\mu^{\pi}$, that is, $\mathbb{E}_{\mu^{\pi}}\left[\cdot\right] := \int (\cdot)\:\differential\mu^{\pi}$. The superscript $\pi$ in $\mu^{\pi}$ indicates that the probability measure depends on the choice of the control policy $\pi$. When the probability measure $\mu^{\pi}$ is absolutely continuous, it admits a PDF $\rho^{\pi}$. The symbol $\sim$ is used as a shorthand for ``follows the probability distribution". 


\section{Problem Formulation}\label{Sec:ProblemFormulation}

\subsection{Colloidal Self-Assembly Sample Path Model}


The specific model we consider for the sample path dynamics for SA of colloidal particles is given by \cite{tang2013colloidal,xue2014optimal} the It\^{o} SDE 
\begin{equation}
\begin{aligned}
\differential \langle C_6\rangle =D_{1}(\langle C_6\rangle,\pi) \: \differential t + \sqrt{2 D_{2}(\langle C_6\rangle,\pi)} \: \differential w
\end{aligned}
\label{SDE1}
\end{equation}
where $t$ denotes time, and the state variable $\langle C_6\rangle\in[0,6]$ is an order parameter denoting the average number of hexagonally close packed particles around each particle. We consider the dynamics \eqref{SDE1} over a fixed time horizon $[0,T]$. The control input $u := \pi(\langle C_6 \rangle, t)\in\mathbb{R}$ denotes electric field voltage \cite{yeh1997assembly} that results from a Markovian control policy $\pi : [0,6] \times [0,T] \mapsto \mathbb{R}$, and $w$ denotes the standard Wiener process in $\mathbb{R}$. In \eqref{SDE1}, the functionals $D_1(\cdot,\cdot)$ and $D_2(\cdot,\cdot)$ are referred to as the \emph{drift} and the \emph{diffusion landscapes}, respectively. In the colloidal SA context, both $D_1,D_2$ are nonlinear in state and non-affine in control. 

Typically, the drift landscape $D_1$ is expressed in terms of the so-called \emph{free energy landscape} $F$ and the \emph{diffusion landscape} $D_2$, as
\begin{align}
    D_1(\langle C_6\rangle,\pi)= &\dfrac{\partial}{\partial \langle C_6\rangle}D_2(\langle C_6\rangle,\pi)\nonumber\\
    &-\frac{D_{2}(\langle C_6\rangle,\pi)}{k_{\rm{B}}  \: \theta }\frac{\partial}{\partial \langle C_6\rangle} F(\langle C_6\rangle,\pi), 
\label{DiffusionFreeEnergy}    
\end{align}
where the Boltzmann constant $k_{\rm{B}} = 1.38066 \times 10^{-23}$ Joules per Kelvin, and $\theta$ denotes a suitable temperature in Kelvin. In Sec. \ref{sec:SimulationResults}, we will give illustrative numerical results for specific choices of the diffusion and the free energy landscapes.

For an admissible Markovian policy $\pi(\cdot, t)$, we assume that the landscapes $D_1,D_2$ satisfy\\
\noindent\textbf{(A1) non-explosion and Lipschitz conditions:} there exist constants $c_1, c_2$ such that $$\|D_1(\cdot,\pi(\cdot, t))\|_{2} + \|\sqrt{D_2\left(\cdot,\pi(\cdot, t)\right)}\|_2 \leq c_1\left(1 + \|\cdot\|_2\right),$$ 
and that $$\|D_1(x,\pi(x, t)) - D_1(\widetilde{x},\pi(\widetilde{x}, t))\|_{2} \leq c_2\|x-\widetilde{x}\|_2$$ for all $x,\widetilde{x}\in[0,6]$, $t\in[0,T]$, \\
\noindent\textbf{(A2) uniformly lower bounded diffusion:} there exists constant $c_3$ such that $\left(\cdot\right)^{\top}D_2\left(\cdot,\pi(\cdot, t)\right)\left(\cdot\right) \geq c_3 \|\cdot\|_2^2$ for all $t\in[0,T]$.\\
\noindent The assumption \textbf{A1} guarantees \cite[p. 66]{oksendal2013stochastic} existence-uniqueness for the sample path of the SDE \eqref{SDE1}. The assumptions \textbf{A1, A2} together guarantee \cite[Ch. 1]{friedman2008partial} that the generator associated with \eqref{SDE1} yields absolutely continuous probability measures $\mu^{\pi}(\langle C_6\rangle,t)$ for all $t>0$ provided the initial measure $\mu_0:=\mu^{\pi}(\cdot,t=0)$ is absolutely continuous. In other words, the PDFs $\rho^{\pi}(\langle C_6\rangle,t)$ exist such that $\differential\mu^{\pi}(\langle C_6\rangle,t) = \rho^{\pi}(\langle C_6\rangle,t)\differential\langle C_6\rangle$. 



\subsection{Controlled Self-Assembly as Distribution Steering }\label{sec:DistributionSteeringProblem}
We propose reformulating the problem of designing a control policy $\pi(\langle C_6\rangle, t)$ for the controlled self-assembly subject to \eqref{SDE1}, to that of steering the \emph{statistics} of the stochastic state $\langle C_6 \rangle$ from a prescribed initial probability measure $\mu_0$ at $t=0$ to a prescribed terminal probability measure $\mu_T$ at $t=T$. This is motivated by the fact that $\langle C_6 \rangle \approx 0$ implies disordered crystalline structure while $\langle C_6 \rangle \approx 5.1$ implies an ordered structure.
So steering the stochastic state from disordered to ordered naturally transcribes to steering a high concentration of probability mass around $0$ to the same around $5.1$. Note the target value of  $\langle C_6 \rangle$ is $5.1$ and not $6$ due to edge effects in the lattice structure.


We emphasize here that we use the term ``statistics" in nonparametric sense, i.e., we allow arbitrary probability measures $\mu_0,\mu_{T}$ supported on the compact set $[0,6]$, and ask for provable steering of $\mu_0$ to $\mu_T$ via control, not just steering of first few statistical moments of $\mu_0$ to $\mu_T$ such as mean and variance. Even if $\mu_0,\mu_T$ have finite dimensional sufficient statistics, the transient $\langle C_6\rangle$ probability measures induced by the nonlinear SDE \eqref{SDE1} for a given control policy $\pi$, may not have so. This motivates formulating the control synthesis as a two point boundary value problem over the (infinite dimensional) manifold of state probability measures.

Specifically, we consider the minimum control effort steering of $\mu_0$ to $\mu_T$, i.e., solving
\begin{equation}
   \begin{aligned}[t] 
\underset{\pi\in\mathcal{U}}{\inf}\:~~~~~~~~~&\int_{0}^{T} \mathbb{E}_{\mu^{\pi}}\left[ \frac{1}{2}\pi^{2}\right]\:\differential t\\
 \text{subject to} ~~~~&\differential \langle C_6\rangle =D_{1}(\langle C_6\rangle,\pi) \differential t + \sqrt{ 2 D_{2}(\langle C_6\rangle,\pi)}\: \differential w,\\
 & \langle C_6\rangle(t=0)\sim \mu_{0}\;\text{(given)}, \; \\
 &\langle C_6\rangle(t=T)\sim \mu_{T}\;\text{(given)},
\label{SBP}    
\end{aligned} 
\end{equation}
where $\mu^{\pi}\equiv\mu^{\pi}(\langle C_6\rangle,t)$ denotes the controlled state probability measure at time $t$. In other words, we design state feedback for dynamically reshaping uncertainties subject to the dynamical constraint \eqref{SDE1}, endpoint statistical constraints, and the deadline constraint.

\begin{figure}[t]
\centering
\includegraphics[width=0.95\linewidth]{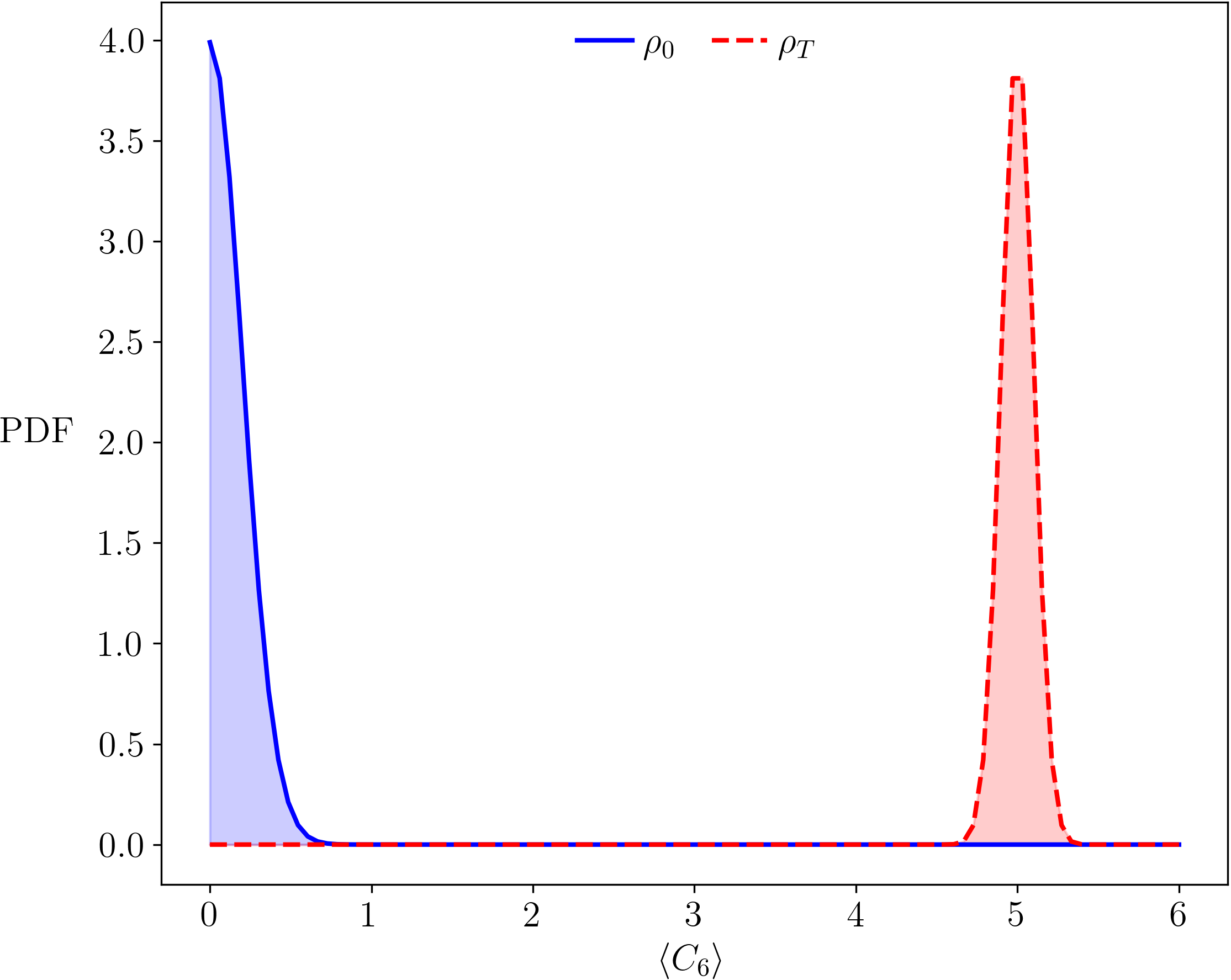}
\caption{{\small{The prescribed initial PDF $\rho_0$ (solid line) at the initial time $t=0$, and the prescribed terminal PDF $\rho_{T}$ (dashed line) at the final time $t=T$. Both PDFs are supported over $[0,6]$, which is the range of values for the state variable $\langle C_6\rangle$ denoting a crystallinity order parameter. In particular, $\langle C_6\rangle \approx 0$ implies a disordered state and $\langle C_6\rangle \approx$ 5-6 implies a highly ordered state.}}}
\vspace*{-0.15in}
\label{fig_Terminal_Distributions}
\end{figure}

In \eqref{SBP}, the set of feasible controls $\mathcal{U}$ comprises of the finite energy Markovian inputs, i.e.,
\begin{align}
\mathcal{U}:=\{\pi : [0,6] \times [0,T] \mapsto \mathbb{R}\mid \int_{0}^{T}\mathbb{E}_{\mu^{\pi}}[\pi^{2}]\: \differential t < \infty\}.
\label{FeasibleControl}    
\end{align}
Recall that the input is the electrical field voltage, and the minimum effort objective is a natural candidate to promote control parsimony. 

We suppose that the endpoint probability measures $\mu_0,\mu_T$ are absolutely continuous with respective PDFs $\rho_0,\rho_T$; see Fig. \ref{fig_Terminal_Distributions}.
Then, we can rewrite \eqref{SBP} as the variational problem:
\begin{subequations}
	\begin{align}
		&\underset{(\rho^{\pi},\pi)}{\inf}\:\int_{0}^{T}\int_{\mathbb{R}}\frac{1}{2}\pi^{2}(\langle C_6\rangle,t)\rho^{\pi}(\langle C_6\rangle,t)\:\differential \langle C_6\rangle\: \differential t \label{SBPobjective} \\
		 &\text{subject to}~~ 	\frac{\partial \rho^{\pi}}{\partial t} =-\frac{\partial}{\partial \langle C_6\rangle}( D_{1}\rho^{\pi})+\frac{\partial^{2}}{\partial \langle C_6\rangle^{2}} (D_{2} \rho^{\pi}),\label{FPKequation}\\
	 &\qquad\qquad\quad\rho^{\pi}(\langle C_6\rangle,0)=\rho_0, \quad \rho^{\pi}(\langle C_6\rangle,T)=\rho_T \label{InitialAndTerminalPDF},
	\end{align}
	\label{Optimization_FPK}
\end{subequations}
where \eqref{FPKequation} is the \emph{controlled} Fokker-Planck-Kolmogorov’s forward (FPK) PDE. The feasible pair $(\rho^{\pi},\pi)\in\mathcal{P}_{0T}\times \mathcal{U}$ where $\mathcal{P}_{0T}$ denotes PDF-valued trajectories connecting $\rho_0,\rho_T$, i.e.,
\begin{align}
\mathcal{P}_{0T} := \big\{&\rho(\cdot,t)\geq 0\mid \int\rho(\cdot,t)\differential(\cdot) = 1 \:\text{for all}\: t\in [0,T],\nonumber\\
&\rho(\cdot,t=0)=\rho_0,\quad \rho(\cdot,t=T)=\rho_T\big\},    
\label{FeasiblePDFvaluedCurves}    
\end{align}
and $\mathcal{U}$ is given by \eqref{FeasibleControl}.

The variational problem \eqref{Optimization_FPK} is an instance of the SBP that concerns with most likely stochastic evolution that transports $\rho_0$ to $\rho_T$ over $[0,T]$. The topic originated in the works of Erwin Schr\"{o}dinger \cite{schrodinger1931umkehrung,schrodinger1932theorie}, and recognizing its connection with stochastic control \cite{jamison1975markov,pra1990markov,dai1991stochastic,chen2021stochastic,chen2021controlling} has unfolded a rapid development in the control literature including when the drift is nonlinear in state \cite{caluya2021wasserstein,caluya2021reflected,nodozi2022schr,caluya2020finite,haddad2020prediction}. The computational approach in these works rely on a contractive fixed point recursion \cite{chen2016entropic} over the so-called Schr\"{o}dinger factors (see e.g., \cite[Sec. II]{caluya2021wasserstein}) resulting from certain change of variables related to the Fleming's logarithmic transform \cite{fleming1982logarithmic,kappen2005path} or the Hopf-Cole transform \cite{hopf1950partial,cole1951quasi}.

With respect to the existing literature on SBP with nonlinear state-dependent drift, the additional difficulty in \eqref{Optimization_FPK} is that $D_1,D_2$ for colloidal SA are \emph{non-affine} in the control policy. To the best of our knowledge, an SBP with this level of generality has not been investigated before. Beyond analytical difficulties, the computation also becomes challenging since non-affine control precludes the aforesaid fixed point recursion approach via the Schr\"{o}dinger factors.

\section{Optimality}\label{secOptimality}
In this work, we will not pursue the existence-uniqueness proofs for the solution of \eqref{Optimization_FPK}. Instead, we next formally derive the first order conditions for optimality for problem \eqref{Optimization_FPK} in the form of three coupled PDEs with endpoint boundary conditions while tacitly assuming the existence-uniqueness. In Sec. \ref{secPINN}, we will numerically solve this system via PINN. 

\begin{figure*}[t]
\centering
\includegraphics[width=.8\paperwidth]{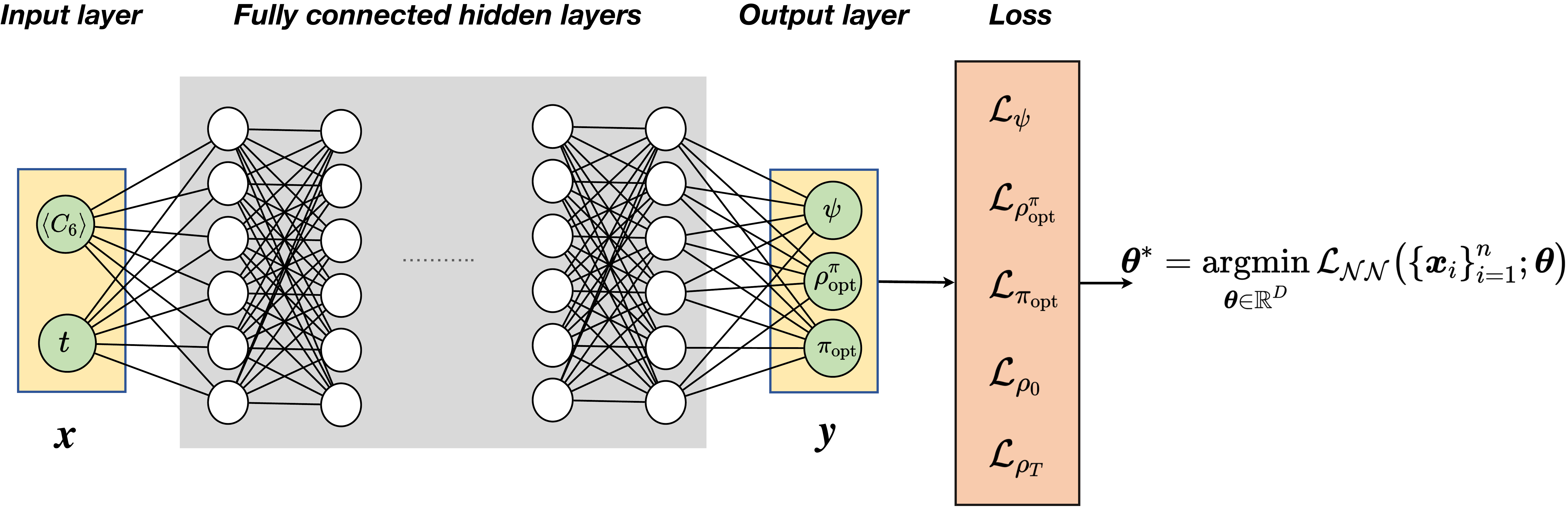}
\caption{{\small{The architecture of the physics-informed neural network with the system order parameter and time as the input features $\bm{x}:=(\langle C_6\rangle, t)$. The output $\bm{y}$ comprises of the value function, optimally controlled PDF, and optimal control policy, i.e., $\bm{y}:=(\psi,\rho^{\pi}_{\rm{opt}}, \pi_{\rm{opt}})$.}}}
\label{fig:PINNStructure}
\end{figure*}

\begin{theorem}\label{theorem1}
    \textbf{(First order conditions for optimality)} The pair $(\rho^{\pi}_{\rm{opt}}(\langle C_6\rangle,t),\pi_{\rm{opt}}(\langle C_6\rangle,t))$ that solves \eqref{Optimization_FPK}, must satisfy the system of coupled PDEs
\begin{subequations}
	\begin{align}
		&\frac{\partial \psi}{\partial t}=\frac{1}{2} (\pi_{\rm{opt}})^{2}-D_{1}\frac{\partial \psi}{\partial \langle C_6\rangle} - D_{2}\frac{\partial^{2}\psi}{\partial \langle C_6\rangle^{2}}, \label{HJB} \\
		 &	\frac{\partial\rho^{\pi}_{\rm{opt}}}{\partial t} =-\frac{\partial }{\partial \langle C_6\rangle}( D_{1}\rho^{\pi}_{\rm{opt}})+\frac{\partial^{2}}{\partial \langle C_6\rangle^{2}}(D_{2} \rho^{\pi}_{\rm{opt}}) ,\label{FPK}\\
		 & \pi_{\rm{opt}}(\langle C_6\rangle,t)=\frac{\partial \psi}{\partial \langle C_6\rangle} \frac{\partial D_{1}}{\partial \pi_{\rm{opt}}} +\frac{\partial^{2} \psi}{\partial \langle C_6\rangle^{2}} \frac{\partial D_{2}}{\partial \pi_{\rm{opt}}},
    \label{OptimalInput}
	\end{align}
	\label{coupled_PDEs}
\end{subequations}
in unknowns $(\psi(\langle C_6\rangle,t),\rho^{\pi}_{\rm{opt}}(\langle C_6\rangle,t),\pi_{\rm{opt}}(\langle C_6\rangle,t))$
with boundary conditions
\begin{equation}
    \rho^{\pi}_{\rm{opt}}(\langle C_6\rangle,0)=\rho_{0}, \; \quad \rho^{\pi}_{\rm{opt}}(\langle C_6\rangle,T)=\rho_{T}.
    \label{boundary_conditions}
\end{equation}
\end{theorem}

\begin{proof}
Consider the Lagrangian associated with \eqref{Optimization_FPK}:
\begin{equation}
\begin{aligned}
        \mathcal{L}(\rho^{\pi},\pi,\psi):=&\int_{0}^{T}\!\!\int_{\mathbb{R}} \left\{ \frac{1}{2}\pi^{2}\rho^{\pi} +\psi\times  \left(  \frac{\partial \rho^{\pi}}{\partial t} \right. \right.\\
        &  \left. \left. +\frac{\partial}{\partial \langle C_6\rangle}( D_{1}\rho^{\pi})-\frac{\partial^{2}}{\partial \langle C_6\rangle^{2}} (D_{2} \rho^{\pi})\right)\right\} \differential x\:\differential t
\end{aligned}
\label{lagrangian}
\end{equation}
where $\psi(\langle C_6\rangle,t)$ is a $C^{1}(\mathbb{R};\mathbb{R}_{\geq0})$ Lagrange multiplier. 

Performing integration by parts, the Lagrangian \eqref{lagrangian} can be written as
\begin{equation}
\begin{aligned}
        \mathcal{L}(\rho^{\pi},\pi,\psi)=\int_{0}^{T}\int_{\mathbb{R}} &\left( \frac{1}{2}\pi^{2} -\frac{\partial \psi}{\partial t} -D_{1} \frac{\partial \psi}{\partial \langle C_6\rangle}\right.\\
        &~~~\left. -D_{2}\frac{\partial^2 }{\partial \langle C_6\rangle^2}\psi \right) \rho^{\pi} \:\differential x\:\differential t.
\end{aligned}
\label{Integrationbyparts}
\end{equation}
For $\rho^{\pi}$ fixed, pointwise minimization of \eqref{Integrationbyparts} with respect to $\pi$  yields \eqref{OptimalInput}.

By substituting  \eqref{OptimalInput} in \eqref{Integrationbyparts} and equating the resulting expression to zero, we get the dynamic programming equation
\begin{equation}
\begin{aligned}
        &\int_{0}^{T}\!\!\int_{\mathbb{R}} \left( \frac{1}{2}\left(\frac{\partial \psi}{\partial \langle C_6\rangle} \frac{\partial D_{1}}{\partial \pi}+\frac{\partial^{2} \psi}{\partial \langle C_6\rangle^{2}} \frac{\partial D_{2}}{\partial \pi}\right)^{2} -\frac{\partial \psi}{\partial t}\right.\\
        & - D_{1} \frac{\partial \psi}{\partial \langle C_6\rangle} -\left.D_{2}\frac{\partial^2 \psi}{\partial \langle C_6\rangle^2} \right) \rho^{\pi}(\langle C_6\rangle,t) \differential \langle C_6\rangle\:\differential t=0.
\end{aligned}
\label{dynamic_programming_equation}
\end{equation}
For \eqref{dynamic_programming_equation} to be satisfied for arbitrary $\rho^{\pi}$, we must have
\begin{align}
\frac{\partial \psi}{\partial t}=&\frac{1}{2} \left(\frac{\partial \psi}{\partial \langle C_6\rangle} \frac{\partial D_{1}}{\partial \pi}+\frac{\partial^{2} \psi}{\partial \langle C_6\rangle^{2}} \frac{\partial D_{2}}{\partial \pi}\right)^{2} \nonumber\\
&-D_{1} \frac{\partial \psi}{\partial \langle C_6\rangle}- D_{2}\frac{\partial^{2} \psi}{\partial \langle C_6\rangle^{2}},
\end{align}
which upon using \eqref{OptimalInput}, gives the  HJB PDE \eqref{HJB}. The associated FPK PDE \eqref{FPK} and the boundary conditions \eqref{boundary_conditions} follow from \eqref{FPKequation} and \eqref{InitialAndTerminalPDF}, respectively.
\end{proof}
\begin{remark}\label{remark_ValueFnHJB}
The equation \eqref{HJB} in Theorem \ref{theorem1} is the HJB PDE, and the variable $\psi(\langle C_6\rangle,t)$ is referred to as the value function.
\end{remark}
\begin{remark}\label{remark_controlnonaffine}
Unlike the conditions for optimality for control-affine SBP \cite[eq. (5.7)-(5.8)]{chen2021stochastic}, \cite[eq. (20)-(21)]{caluya2021wasserstein}, \cite[eq. (4)]{caluya2021reflected} where we get two coupled PDEs, one being HJB and another being FPK, the system of equations \eqref{coupled_PDEs} has three coupled PDEs because the policy equation \eqref{OptimalInput} is implicit in $\pi_{\rm{opt}}$. Due to the non-affine control, we can no longer express $\pi_{\rm{opt}}$ as the scaled gradient of the value function $\psi$. Instead, we now need to solve the coupled system \eqref{coupled_PDEs}-\eqref{boundary_conditions}.
\end{remark}


\section{Solving the Conditions for Optimality \\using PINN}\label{secPINN}
We propose leveraging recent advances in neural network based computational frameworks to jointly learn the solutions of \eqref{coupled_PDEs}-\eqref{boundary_conditions}. In the following, we discuss training of a PINN \cite{raissi2019physics,lu2021deepxde} to numerically solve \eqref{coupled_PDEs}-\eqref{boundary_conditions}, which is a system of three coupled PDEs together with the endpoint PDF boundary conditions.

The proposed architecture of the PINN is shown in  Fig. \ref{fig:PINNStructure}. In our problem, $\bm{x}:=(\langle C_6\rangle, t)$ comprises of the features given to the DNN, and the DNN output $\bm{y}:=(\psi,\rho^{\pi}_{\rm{opt}}, \pi_{\rm{opt}})$ comprises of the value function, optimally controlled PDF, and optimal policy. We parameterize the output of the fully connected feed-forward DNN via $\bm{\theta} \in \mathbb{R}^{D}$, i.e.,
\begin{align}
\bm{y}(\bm{x}) \approx \mathcal{NN}(\bm{x} ; \bm{\theta}),
\end{align}
where $\mathcal{NN}(\cdot;\bm{\theta})$ denotes the neural network approximant parameterized by $\bm{\theta}$, and $D$ is the dimension of the parameter space (i.e., the total number of to-be-trained weight, bias and scaling parameters for the DNN).

The overall loss function for the network, denoted as $\mathcal{L}_{\mathcal{NN}}$, consists of the sum of the losses associated with the three equations in \eqref{coupled_PDEs} and the losses associated with the boundary conditions \eqref{boundary_conditions}. Specifically, let $\mathcal{L}_{\psi}$ be the loss term for the HJB PDE \eqref{HJB}, let $\mathcal{L}_{\rho^{\pi}_{\rm{opt}}}$ be the loss term for the FPK PDE \eqref{FPK}, and let $\mathcal{L}_{\pi_{\rm{opt}}}$ be the loss term for the control policy equation \eqref{OptimalInput}. Likewise, let  $\mathcal{L}_{\rho_{0}}$ and  $\mathcal{L}_{\rho_{T}}$ denote the loss terms for the corresponding endpoint constraints \eqref{boundary_conditions}. Then
\begin{align}
\mathcal{L}_{\mathcal{NN}}:=\mathcal{L}_{\psi}+\mathcal{L}_{\rho^{\pi}_{\rm{opt}}}+\mathcal{L}_{\pi_{\rm{opt}}}+\mathcal{L}_{\rho_{0}}+\mathcal{L}_{\rho_{T}},
\label{totalloss}
\end{align} 
where each summand loss term in \eqref{totalloss} is evaluated on a set of $n$ collocation points $\{\bm{x}_i\}_{i=1}^{n}$ in the domain of the feature space $\Omega:=[0,6]\times[0,T]$, i.e., $\{\bm{x}_i\}_{i=1}^{n}\subset\Omega$, and
\begin{align*}
\mathcal{L}_{\psi} &:=\frac{1}{n} \sum_{i=1}^{n}\left(\left.\frac{\partial \psi}{\partial t}\right|_{\bm{x}_{i}}-\left.\frac{1}{2} (\pi_{\rm{opt}})^{2}\right|_{\bm{x}_{i}}\left.+D_{1}\frac{\partial }{\partial \langle C_6\rangle}\psi \right|_{\bm{x}_{i}}\right.\\
&\qquad \qquad\qquad\qquad\qquad\qquad\left.\left.+D_{2}\frac{\partial^{2} }{\partial \langle C_6\rangle^{2}}\psi \right|_{\bm{x}_{i}}\right)^{2}, \\
\mathcal{L}_{\rho^{\pi}_{\rm{opt}}} &:=\frac{1}{n} \sum_{i=1}^{n}\left(\left.	\frac{\partial \rho^{\pi}_{\rm{opt}}}{\partial t}\right|_{\bm{x}_{i}}+\left.\frac{\partial }{\partial \langle C_6\rangle}( D_{1}\rho^{\pi}_{\rm{opt}})\right|_{\bm{x}_{i}}\right.\\
&\qquad \qquad\qquad\qquad\qquad\left.\left.-\frac{\partial^{2}}{\partial \langle C_6\rangle^{2}}(D_{2} \rho^{\pi}_{\rm{opt}})\right|_{\bm{x}_{i}}\right)^{2},\\
\mathcal{L}_{\pi_{\rm{opt}}} &:=\frac{1}{n} \sum_{i=1}^{n}\left(\left.	\pi_{\rm{opt}}\right|_{\bm{x}_{i}}-\left.\frac{\partial }{\partial \langle C_6\rangle}\psi \frac{\partial }{\partial \pi_{\rm{opt}}}D_{1}\right|_{\bm{x}_{i}}\right.\\
&\qquad \qquad\qquad\qquad\qquad\left.\left.-\frac{\partial^{2} }{\partial \langle C_6\rangle^{2}}\psi \frac{\partial }{\partial \pi_{\rm{opt}}}D_{2}\right|_{\bm{x}_{i}}\right)^{2},\\
\mathcal{L}_{\rho_{0}}&:=\frac{1}{n} \sum_{i=1}^{n}\left(\left.\rho^{\pi}_{\rm{opt}}(\cdot,t=0)\right|_{\bm{x}_i}-\left.\rho_{0}(\cdot)\right|_{\bm{x}_i}\right)^{2},\\
\mathcal{L}_{\rho_{T}}&:=\frac{1}{n} \sum_{i=1}^{n}\left(\left.\rho^{\pi}_{\rm{opt}}(\cdot,t=T)\right|_{\bm{x}_i}-\left.\rho_{T}(\cdot)\right|_{\bm{x}_i}\right)^{2},
\end{align*}
for each collocation point $\bm{x}_i$, $i=1,\hdots,n$.

For training the PINN, we minimize the overall loss \eqref{totalloss} over $\bm{\theta}\in\mathbb{R}^{D}$ by solving
\begin{align}
\bm{\theta}^{*}=\underset{\bm{\theta} \in \mathbb{R}^{D}}{\operatorname{argmin}}\: \mathcal{L}_{\mathcal{NN}}(\{\bm{x}_i\}_{i=1}^{n}; \bm{\theta}).    
\label{NNtraining}    
\end{align}
In the next Section, we detail the simulation set up and report the numerical results.
\begin{center}
\begin{figure}[t]
\centering
\includegraphics[width=\linewidth]{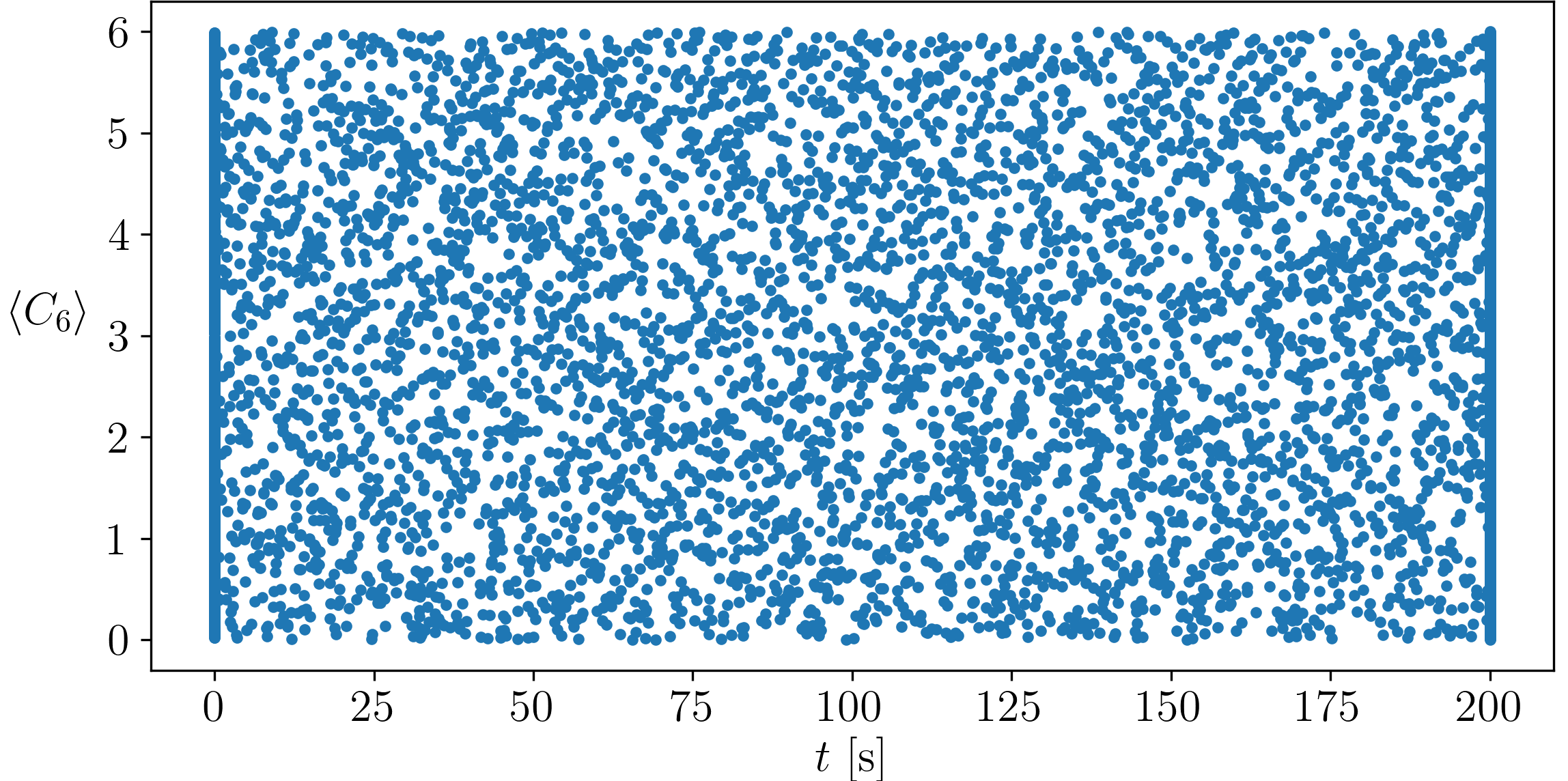}
\caption{{\small{Training data in the domain $\Omega =[0,6]\times[0,200]$.}}}
\label{fig:points}
\end{figure}    
\end{center}

\begin{figure*}[t]
\centering
\includegraphics[width=.65\paperwidth]{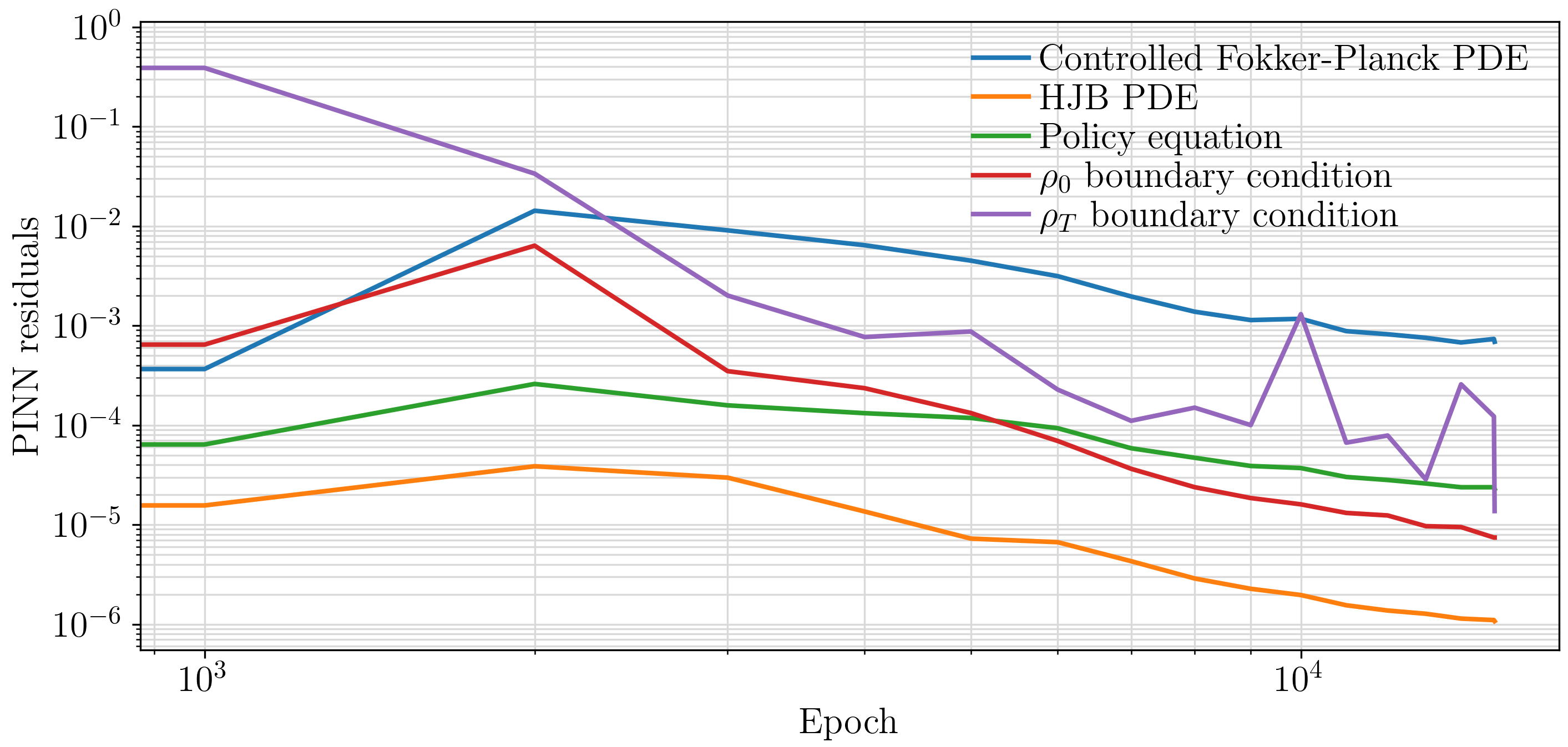}
\caption{{\small{The PINN residuals for solving the conditions for optimality \eqref{coupled_PDEs}-\eqref{boundary_conditions}.}}}
\label{fig:PINNResiduals}
\end{figure*}

\section{Simulation Results}\label{sec:SimulationResults}
We consider the colloidal SA from \cite{tang2013colloidal,xue2014optimal}, where the free energy and the diffusion landscapes are 
\begin{equation}
F(\langle C_6\rangle,\pi)=a\:k_{\rm{B}}\: \theta\left(\langle C_6\rangle-b-c\pi\right)^2, \label{F}
\end{equation}
\begin{equation}
D_{2}(\langle C_6\rangle,\pi)=d \: \exp\left(-(\langle C_6\rangle-b-c\pi)^2\right)+f, 
\label{D_2}
\end{equation}
with known parameters $a=10, b=2.1, c=0.75, d=4.5 \times 10^{-3},f=0.5 \times 10^{-3}$, $k_{\rm{B}} = 1.38066 \times 10^{-23}$ Joules per Kelvin, and $\theta = 293$ Kelvin.

We use the DeepXDE library \cite{lu2021deepxde} for training the PINN. In particular, we choose a neural network with 3 hidden layers with 70 neurons in each layer. The activation functions are chosen to be $\tanh(\cdot)$. The input-output structure of the network are as explained in Sec. \ref{secPINN}. For solving \eqref{NNtraining}, we use the Adam optimizer \cite{kingma2014adam}.

We fix the final time $T=200$ s, and consider the endpoint PDFs $\rho_0,\rho_T$ shown in Fig. \ref{fig_Terminal_Distributions}, represented as truncated normal PDFs (see e.g., \cite[Sec. 2.2]{robert1995simulation}): 
$$\rho_i(x) := \!\!\begin{cases}
\dfrac {1}{\sigma_i}\,{\dfrac {\phi \left({\frac {x-\mu_i }{\sigma_i }}\right)}{\Phi \left({\dfrac {b-\mu_i }{\sigma_i }}\right)-\Phi \left({\dfrac {a-\mu_i }{\sigma_i }}\right)}}\!\!\!\! & \text{for}\quad\!\!\!\! a\leq x \leq b,\\
0 & \text{otherwise},
\end{cases}
$$
where $i\in\{0,T\}$, $\mu_0 = 0, \mu_T = 5, \sigma_0 = 0.2, \sigma_T = 0.1, a=0, b=6$. The functions $\phi(\cdot)$ and $\Phi(\cdot)$ denote the standard normal PDF, and the standard normal cumulative distribution function, respectively. Recall from Sec. \ref{sec:DistributionSteeringProblem} that our proposed method is applicable for arbitrary compactly supported endpoint PDFs.

As shown in Fig. \ref{fig:points}, we choose 1000 training points at each of the initial ($t=0$ s) and terminal ($t=200$ s) times, and another 5000 state-time points inside the domain $\Omega:=[0,6]\times[0,200]$. We used the residual-based adaptive refinement method in \cite{lu2021deepxde} for generating these training points. As shown in Fig. \ref{fig:PINNResiduals}, after 15,000 training epochs, the residuals for all loss functions go below $10^{-3}$. The training was performed in Python 3 on a MacBook Air with 1.1 GHz Quad-Core Intel Core i5 processor and 8 GB memory. The computational time for training was 2097.40 seconds.

\begin{figure}[t]
\centering
\includegraphics[width=1\linewidth]{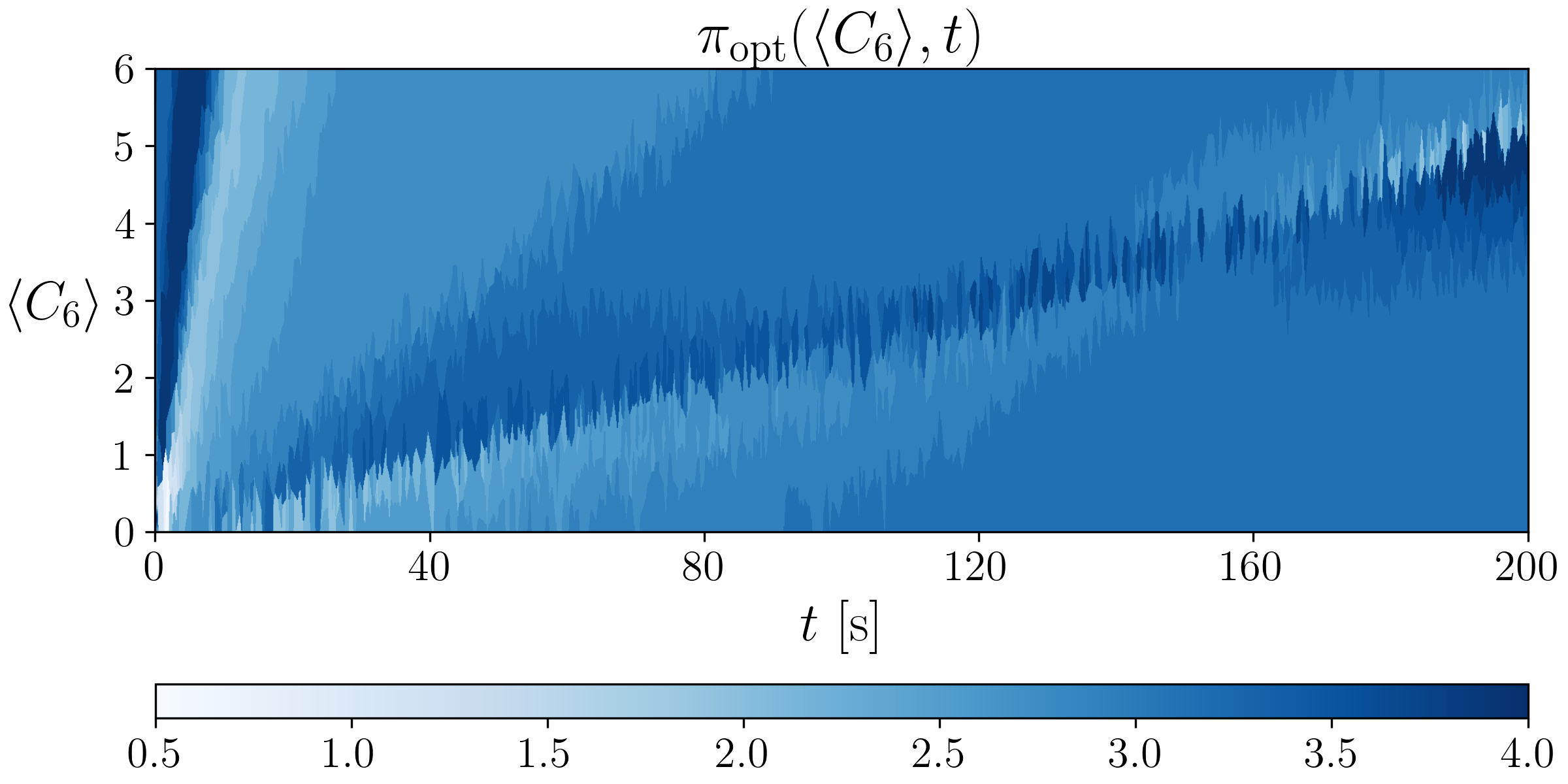}
\caption{{\small{The optimal policy $\pi_{\rm{opt}}\left(\langle C_6\rangle,t\right)$ obtained as an output of the trained PINN solving the conditions for optimality \eqref{coupled_PDEs}-\eqref{boundary_conditions}.}}}
\label{fig:OptimalPolicy}
\end{figure}

Fig. \ref{fig:OptimalPolicy} shows the optimal policy $\pi_{\rm{opt}}\left(\langle C_6\rangle,t\right)$ obtained as the output of the trained PINN. The value function $\psi\left(\langle C_6\rangle,t\right)$ obtained as another output of the trained PINN is shown in Fig. \ref{fig:ValueFunction}. In Fig. \ref{fig:OptimallyControlledJointPDFs}, the snapshots of the optimally controlled PDFs $\rho_{\rm{opt}}^{\pi}$ obtained from the trained PINN are shown as solid curves with grey filled areas. To verify these results, we sampled 1000 initial states from the given $\rho_0$ using the Metropolis-Hastings \cite{hastings1970monte} Markov Chain Monte Carlo algorithm, and then performed closed-loop simulations using the optimal policy $\pi_{\rm{opt}}\left(\langle C_6\rangle,t\right)$ (shown in Fig. \ref{fig:OptimalPolicy}) obtained from the PINN. The stem plots shown in Fig. \ref{fig:OptimallyControlledJointPDFs} depict the kernel density estimates (KDE) of these closed-loop trajectories. The KDE stems provide empirical approximations for the closed-loop optimally controlled PDFs, which match very well with the learnt solutions from PINN.

\begin{figure}[t]
\centering
\includegraphics[width=1\linewidth]{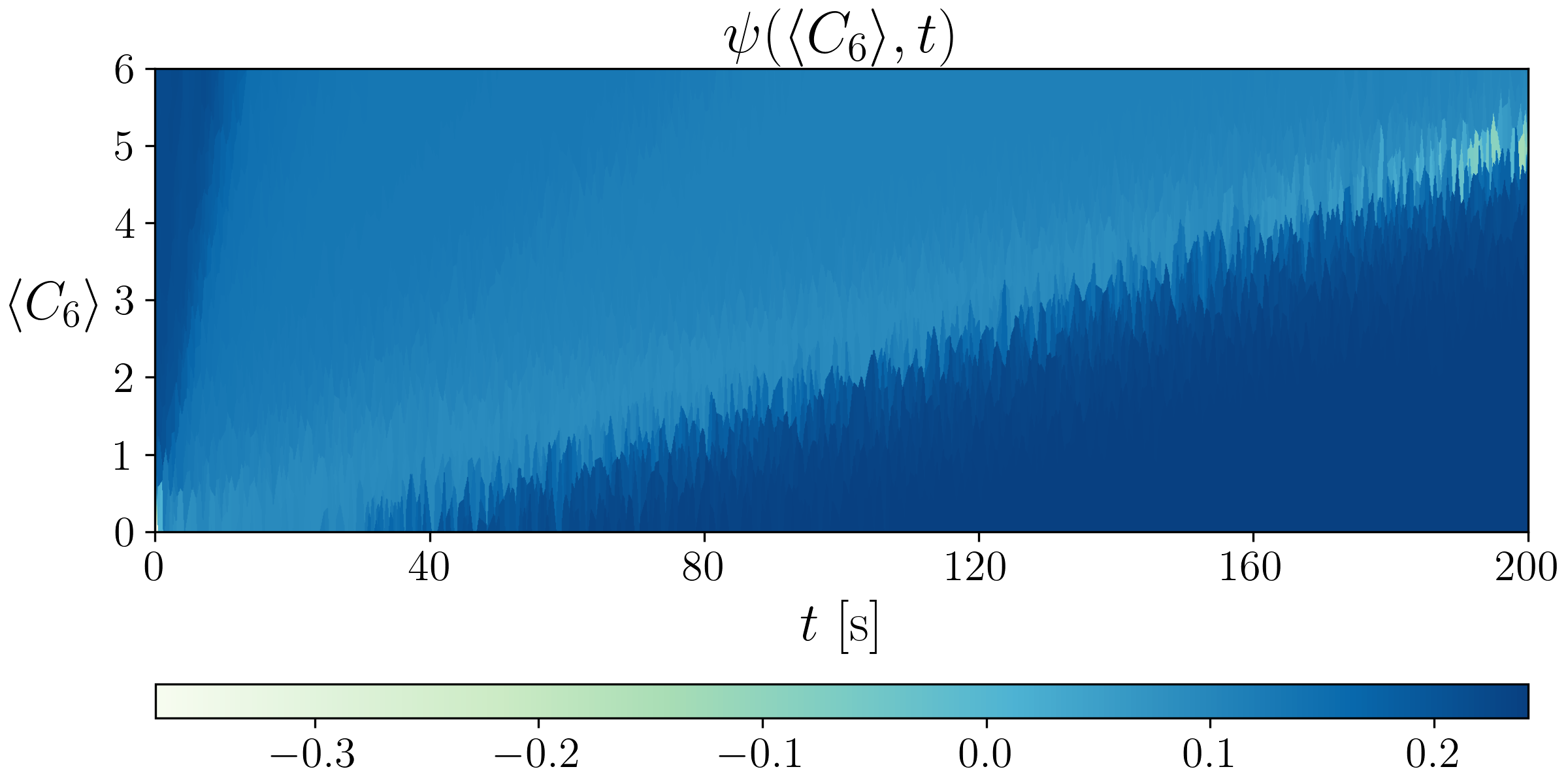}
\caption{{\small{The value function $\psi\left(\langle C_6\rangle,t\right)$ obtained as an output of the trained PINN solving the conditions for optimality \eqref{coupled_PDEs}-\eqref{boundary_conditions}.}}}
\label{fig:ValueFunction}
\end{figure}

\begin{figure*}[t]
\centering
\includegraphics[width=.65\linewidth]{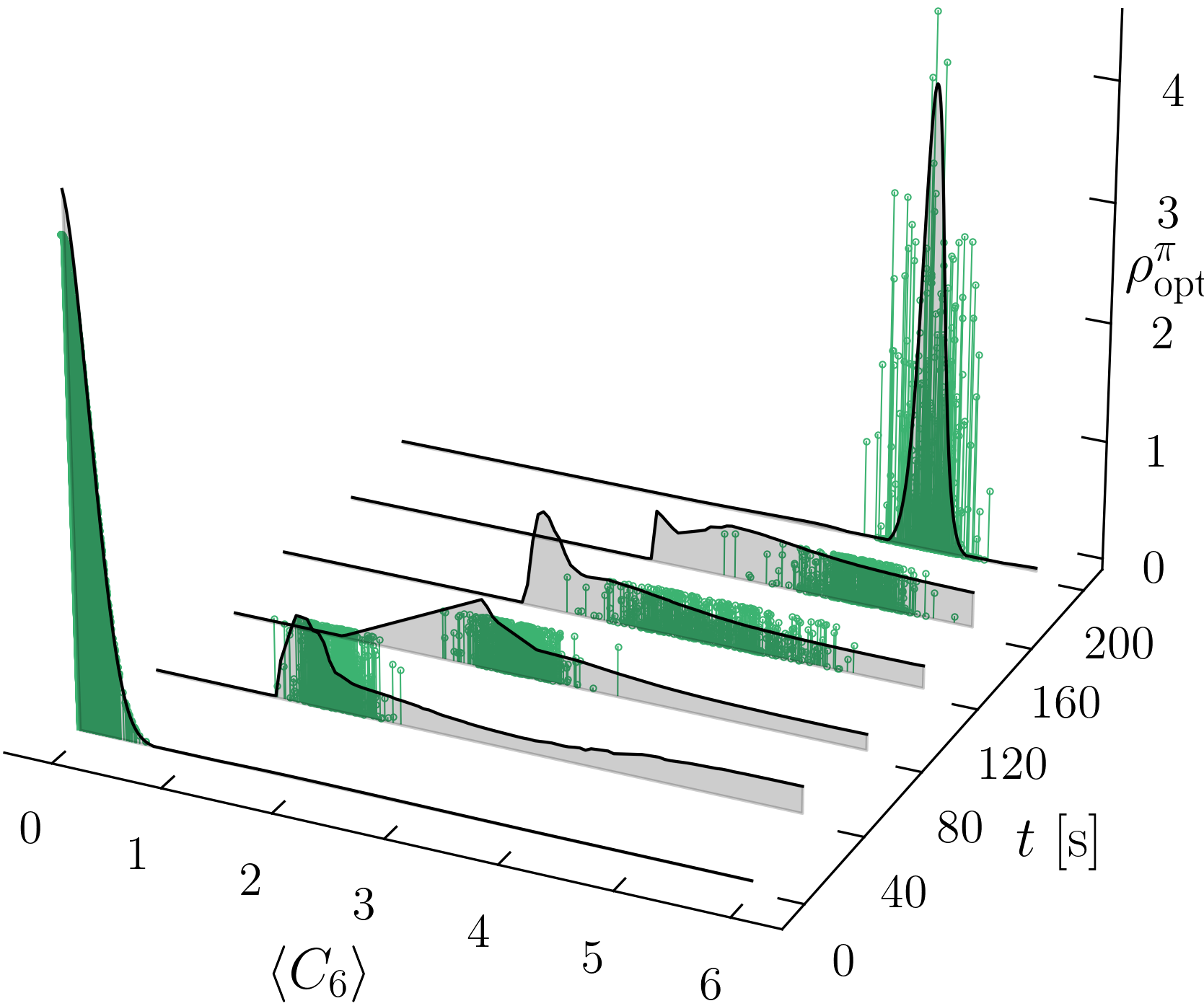}
\caption{{\small{Snapshots of the optimally controlled joint PDFs $\rho^{\pi}_{\rm{opt}}$ steering the state $\langle C_6\rangle$ distribution from the given $\rho_{0}$ to $\rho_{T}$, as in Fig. \ref{fig_Terminal_Distributions}, over the given time horizon $[0,T]\equiv [0,200]$ s subject to the controlled noisy nonlinear sample path dynamics \eqref{SDE1}. The solid black curves with grey filled areas are obtained from the PINN. The stem plots are the KDE approximants of the optimally controlled PDF snapshots obtained from the closed-loop sample paths, as explained in Sec. \ref{sec:SimulationResults}.}}}
\label{fig:OptimallyControlledJointPDFs}
\end{figure*}

Fig. \ref{fig:SamplePath} shows the 1000 random sample paths for the closed loop simulation using the learnt optimal policy $\pi_{\rm{opt}}\left(\langle C_6\rangle,t\right)$ that provably steers the given $\rho_0$ from $t=0$ to the given $\rho_T$ at $t=T$ over $[0,T]\equiv [0,200]$ s.

\begin{figure*}[t!]
    \centering
    \begin{subfigure}[t]{0.49\textwidth}
        \centering
        \includegraphics[height=2.7in]{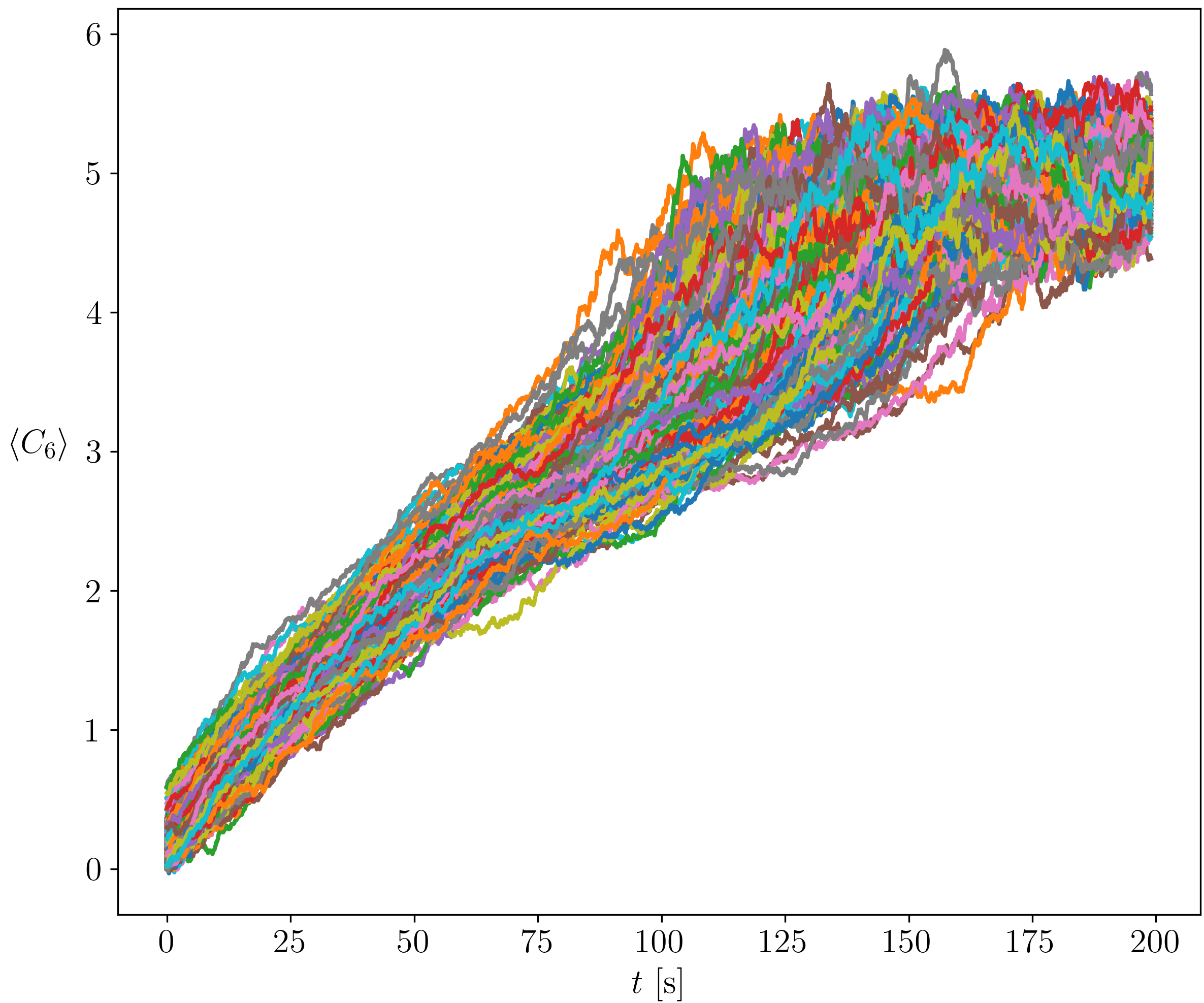}
       \caption{{\small{Optimally controlled $\langle C_6\rangle$ state trajectories.}}}
       \label{fig:OptimalClosed-loopStateTrajectories}
    \end{subfigure}%
    ~ 
    \begin{subfigure}[t]{0.49\textwidth}
        \centering
        \includegraphics[height=2.7in]{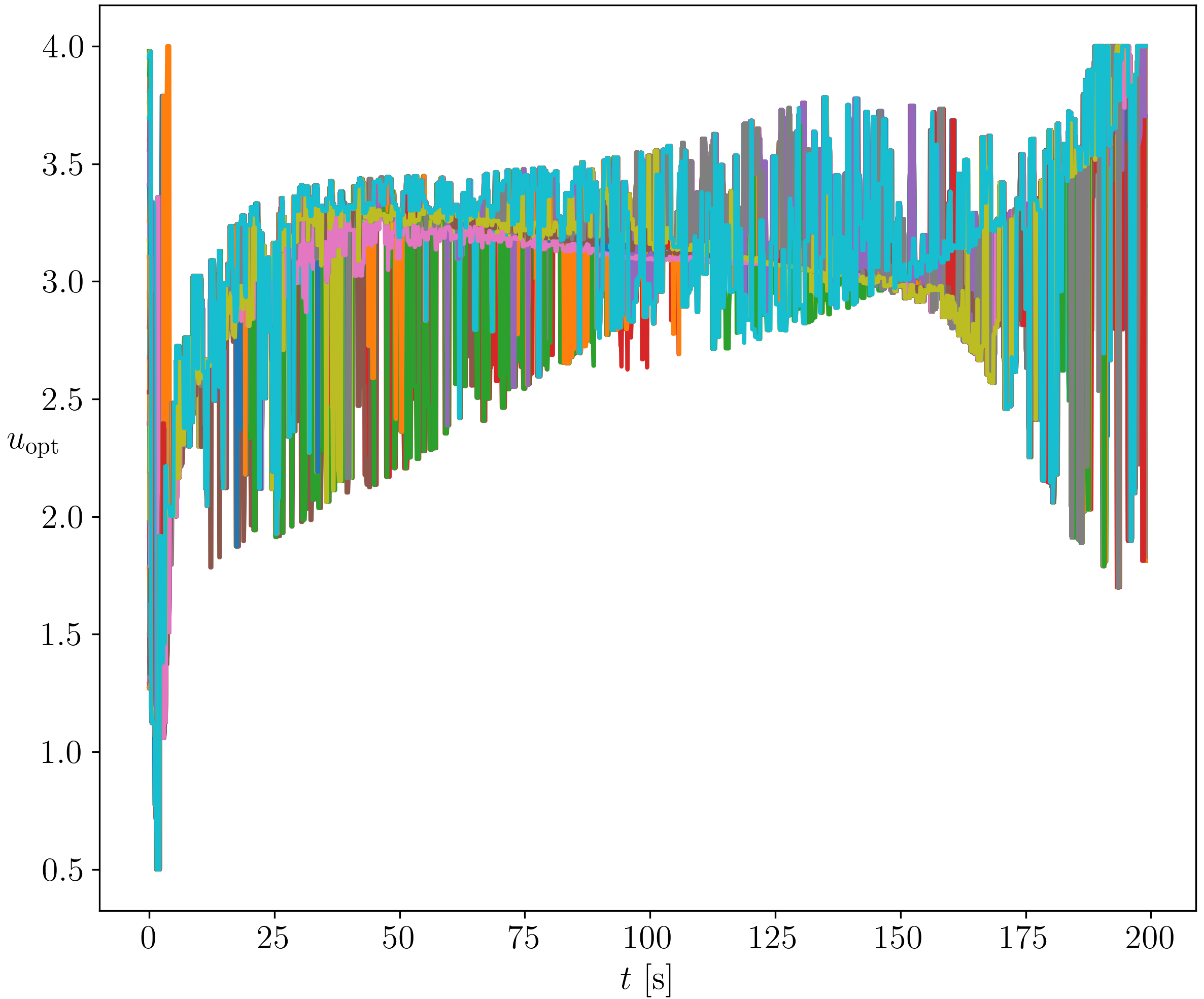}
       \caption{{\small{Optimal control $u_{\rm{opt}}$ trajectories.}}}
       \label{fig:OptimalControlTrajectories}
    \end{subfigure}
    \caption{{\small{The 1000 random sample paths resulting from the closed loop simulation using the learnt optimal policy $\pi_{\rm{opt}}\left(\langle C_6\rangle,t\right)$.}}}
    \label{fig:SamplePath}
\end{figure*}

\section{Conclusions}\label{sec:conclusions}
We propose formulating the finite horizon stochastic control problem for colloidal SA as a two point boundary value problem in the space of PDFs supported over the underlying state space. We develop the idea in detail for a univariate state model from the literature. We show that the resulting problem leads to a ``nonlinear in state and non-affine in control" variant of the Schr\"{o}dinger bridge problem (SBP), and derive the conditions for optimality for the same. We point out how the resulting system of equations differ from the existing SBP literature, and the difficulty in applying the standard computational approach of solving the SBP via fixed point recursions in our setting. These difficulties are fundamentally caused by both the drift and diffusion landscapes being
control non-affine--a situation typical for the colloidal SA application. To circumvent these challenges, we adopt a learning based approach, and train a PINN to jointly learn the optimal control policy, the value function and the optimally controlled state PDFs. Our numerical experiments on a simple benchmark from the literature show that the proposed method performs very well.

The technical proof for the existence-uniqueness of the generalized SBP solutions will appear in our follow up work. Future work will also incorporate data-driven high dimensional colloidal SA models obtained from high-fidelity molecular dynamics simulation data.
\section*{Acknowledgment}
We are indebted to Lu Lu for helpful discussions on the DeepXDE toolbox \cite{lu2021deepxde}.
\balance
\bibliographystyle{IEEEtran}
\bibliography{References.bib}

\end{document}